\documentclass[a4paper]{amsart}
\usepackage{amssymb} 
\usepackage{stmaryrd}
\usepackage[mathcal]{euscript}
\usepackage{tikz}
\usepackage{tikz-cd}
\usepackage{hyperref}
\hypersetup{%
  bookmarksnumbered=true,%
  colorlinks=true,%
  linkcolor=blue,%
  citecolor=blue,%
  filecolor=blue,%
  menucolor=blue,%
  urlcolor=blue,%
  bookmarksopen=true,%
  bookmarksdepth=2,%
  pageanchor=true}

\makeatletter
\@namedef{subjclassname@2020}{\textup{2020} Mathematics Subject Classification}
\makeatother

\title{Tilting preserves finite global dimension}

\author[Bernhard Keller]{Bernhard Keller}
\author[Henning Krause]{Henning Krause}

\address{Universit\'e Paris Diderot -- Paris~7, UFR de
  Math\'ematiques, Institut de Math\'ematiques de Jussieu--PRG, UMR
  7586 du CNRS, Case 7012, B\^atiment Sophie Germain, 75205 Paris
  Cedex 13, France}
\email{bernhard.keller@imj-prg.fr}

\address{Fakult\"at f\"ur Mathematik,
Universit\"at Bielefeld, 33501 Bielefeld, Germany}
\email{hkrause@math.uni-bielefeld.de}

\theoremstyle{plain}
\newtheorem{thm}{Theorem}

\newtheorem{lem}[thm]{Lemma}

\theoremstyle{definition}

\theoremstyle{remark}
\newtheorem{rem}[thm]{Remark}

\numberwithin{equation}{section}

\newcommand{\add}{\operatorname{add}}

\newcommand{\colim}{\operatorname*{colim}}

\newcommand{\End}{\operatorname{End}}

\newcommand{\Ext}{\operatorname{Ext}}

\newcommand{\gldim}{\operatorname{gl{.}dim}}

\newcommand{\Hom}{\operatorname{Hom}}

\newcommand{\Inj}{\operatorname{Inj}}

\newcommand{\Lcolim}{\operatorname*{\mathbf{L}colim}}

\newcommand{\Lex}{\operatorname{Lex}}

\renewcommand{\mod}{\operatorname{mod}}
\newcommand{\Mod}{\operatorname{Mod}}

\newcommand{\pcoh}{\operatorname{pcoh}}

\newcommand{\proj}{\operatorname{proj}} 
\newcommand{\Proj}{\operatorname{Proj}}

\newcommand{\RHom}{\operatorname{\mathbf{R}Hom}}
\newcommand{\Rlim}{\operatorname*{\mathbf{R}lim}}

\newcommand{\Ab}{\mathrm{Ab}}

\newcommand{\op}{\mathrm{op}}

\newcommand{\col}{\colon}

\newcommand{\iso}{\xrightarrow{\raisebox{-.4ex}[0ex][0ex]{$\scriptstyle{\sim}$}}}

\newcommand{\longiso}{\xrightarrow{\ \raisebox{-.4ex}[0ex][0ex]{$\scriptstyle{\sim}$}\ }}

\newcommand{\lto}{\longrightarrow}

\newcommand*{\intref}[2]{\def\tmp{#1}\ifx\tmp\empty\hyperref[#2]{\ref*{#2}}\else\hyperref[#2]{#1~\ref*{#2}}\fi}

\def\A{\mathcal A}

\def\D{\mathcal D}

\def\T{\mathcal T} 
\def\U{\mathcal U}
\def\V{\mathcal V}

\def\bfi{\mathbf i}

\def\bfL{\mathbf L}
\def\bfR{\mathbf R}

\def\bfD{\mathbf D} 
\def\bfK{\mathbf K}
\def\bfL{\mathbf L}

\def\bbN{\mathbb N}

\def\bbZ{\mathbb Z}

\def\e{\varepsilon}

\def\t{\tau}

\def\La{\Lambda}
\def\Si{\Sigma}

\begin{document}

\begin{abstract}
  Given a tilting object of the derived category of an abelian
  category of finite global dimension, we give (under suitable
  finiteness conditions) a bound for the global dimension of its
  endomorphism ring.
\end{abstract}



\date{27 April, 2020}
\maketitle

\section*{Introduction}

Tilting theory \cite{AHK2007} allows us to construct derived equivalences
in various settings. Prime examples are the derived equivalences
between algebras obtained from tilting modules \cite{Ha1988} or
tilting complexes \cite{Ri1989} and the derived equivalences between
algebras and (non commutative) varieties obtained from tilting
bundles, cf.\ for example \cite{Be1978, Kapranov88, GeigleLenzing85,
  Baer88}. An important consequence of the existence of a derived
equivalence is the agreement of various subordinate invariants.  For
instance, the Grothendieck group \cite{Ri1989} and Hochschild
cohomology \cite{Ha1989,Ri1991,Keller04} are preserved.  Another invariant is
the finiteness of global dimension, to which this note is devoted. It
is well-known that finiteness of global dimension is preserved when
two algebras are linked by a tilting module \cite[III.3.4]{Ha1988} or
a tilting complex \cite[12.5]{GaRo92}. Similar facts hold in the
geometric examples.  It seems natural to unify the algebraic and
geometric examples by considering the following general question:

 \emph{Given a
  tilting object $T$ in the (bounded) derived category of an abelian
  category $\A$, does finite global dimension of $\A$ imply finite
  global dimension of the endomorphism ring of $T$?}

Despite the ubiquity of tilting objects in algebra and geometry, there
seems to be no general result in the literature which guarantees that
tilting preserves finite global dimension, even when the category $\A$
is hereditary.\footnote{Theorem~6.1 in \cite{Le2007} claims that
  $\End(T)$ has finite global dimension when $\A$ is hereditary, but
  the proof seems to be incomplete.}  An explanation may be possible
confusion about the very definition of a tilting object. In fact,
there are various possible definitions in the literature, and we need
to clarify this point.

Let $\A$ be an abelian category. By definition, its \emph{global dimension}
is the infimum of the integers $d$ such that $\Ext_\A^i(-,-)=0$ for all $i>d$. Denote by
$\bfD(\A)$ the derived category of $\A$. Fix an object $T\in\bfD(\A)$ and set
$\La=\End(T)$.   We assume that $\Hom(T,\Si^i T)=0$ for all $i\neq 0$.

We consider \emph{two settings} for $T$ to be a tilting object,
depending on whether the abelian category $\A$ is essentially small or
not. For the first setting, we focus on the bounded derived category
$\bfD^b(\A)$ of objects with cohomology concentrated in finitely many
degrees.  Then we define $T\in\bfD^b(\A)$ to be \emph{tilting} if
$\bfD^b(\A)$ equals the thick subcategory generated by
$T$.\footnote{Often the following weaker condition is used:
  $\Hom(T,\Si^i X)=0$ for all $i\in\bbZ$ implies $X=0$. This is not
  sufficient in our context.}  For example, if $\Gamma$ is a right coherent
ring of finite global dimension and $\A$ the abelian category
$\mod\Gamma$ of finitely presented right $\Gamma$-modules, then the
object $T$ of $\bfD^b(\A)$ is tilting if and only if it is isomorphic
to a tilting complex in the sense of \cite{Ri1989}.

\begin{thm}\label{th:main}
  Let $T\in\bfD^b(\A)$ be tilting. Suppose that $\A$ is noetherian,
  that is, each object in $\A$ is noetherian.  Then the global
  dimension of $\La$ is at most $2d+t$, where $d$ is the global
  dimension of $\A$ and $t$ the smallest integer such that $H^iT=0$
  for all $i$ outside an interval of length $t$. Moreover,
  $\RHom(T,-)$ induces a triangle equivalence
  $\bfD^b(\A)\iso\bfD^b(\mod\La)$ when $\La$ is right coherent.
\end{thm}

For our second setting, assume that $\A$ is a Grothendieck category so
that $\bfD(\A)$ has arbitrary (set-indexed) coproducts given by
coproducts of complexes. Recall that an object $C$ of $\bfD(\A)$ is
called \emph{compact} if the functor $\Hom(C,-)$ commutes with
arbitrary coproducts. Each compact object lies in $\bfD^b(\A)$,
cf.~\intref{Lemma}{lemma:compact}. Then we define $T\in\bfD(\A)$ to be
\emph{tilting} if it is compact and $\bfD(\A)$ equals the the
localizing subcategory generated by $T$ (the closure under
$\Sigma^{\pm 1}$, extensions and arbitrary coproducts). For example,
if $\A$ is the category $\Mod \Gamma$ of all right modules over a ring
$\Gamma$, then the tilting objects in $\bfD(\A)$ are precisely those
isomorphic to tilting complexes in the sense of \cite{Ri1989}.

\begin{thm}\label{th:main2}
  Let $T\in\bfD(\A)$ be tilting.  Then $\RHom(T,-)$ induces a triangle
  equivalence $\bfD(\A)\iso\bfD(\Mod\La)$ and $\gldim\La\le 2d+t$,
  where $d$ and $t$ are defined as in \intref{Theorem}{th:main}.
\end{thm}

We deduce \intref{Theorem}{th:main} from \intref{Theorem}{th:main2}.
The proof uses t-structures and the strategy is inspired by
\cite[12.5]{GaRo92}.  For \intref{Theorem}{th:main2}, we compare the
canonical t-structure on $\bfD(\A)$ with the canonical one on
$\bfD(\Mod\La)$; this yields the bound for the global dimension of
$\La$.  For \intref{Theorem}{th:main}, we embed $\A$ into a
Grothendieck category $\bar\A$ and employ the fact that a tilting
object $T\in\bfD^b(\A)$ identifies with a tilting
object in the unbounded derived category $\bfD(\bar\A)$.

\section*{t-structures and finite global dimension}

Let $\T$ be a triangulated category with suspension
$\Si\colon\T\iso\T$. A pair $(\U,\V)$ of full additive subcategories
is called \emph{t-structure} provided the following holds
\cite{BBD1982}:
\begin{enumerate}
\item $\Si \U\subseteq\U$ and  $\Si^{-1} \V\subseteq\V$.
\item $\Hom(X,Y)=0$ for all $X\in\U$ and $Y\in\V$. 
\item For each $X\in\T$ there exists an exact triangle
  $X'\to X\to X''\to\Si X'$ such that $X'\in\U$ and $X''\in\V$.
\end{enumerate}

We consider the following example. Let $\A$ be an abelian category and
$\T=\bfD(\A)$ its derived category. For $n\in\bbZ$ set
\[\T^{\le n}:=\{X\in\T\mid H^iX=0\textrm{ for all }i > n\},\]
and
\[\T^{> n}:=\{X\in\T\mid H^iX=0\textrm{ for all }i \le n\}.\]
Then we have $\T^{\le n}=\Si^{-n}\T^{\le 0}$ and
$\T^{> n}=\Si^{-n}\T^{> 0}$ for all $n\in\bbZ$. For each $X\in\T$ the
truncations in degree $n$ provide an exact triangle
\[\t_{\le n}X\lto X\lto\t_{>n}X \lto \Si (\t_{\le n}X)\]
with $\t_{\le n}X\in \T^{\le n}$ and $\t_{> n}X\in \T^{> n}$. Thus the
pair $(\T^{\le 0},\T^{>0})$ is a t-structure and called
\emph{canonical t-structure} on $\bfD(\A)$. Note that the canonical
t-structure restricts to the one on $\bfD^b(\A)$.

\begin{lem}\label{le:can-t-structure-gldim}
  Let  $(\D^{\le 0},\D^{>0})$ denote the canonical t-structure on
  $\bfD^b(\A)$. Then the global dimension of $\A$ is bounded by $d$ if
  and only if $\Hom(X,Y)=0$ for all $X\in\D^{\ge 0}$ and
  $Y\in\D^{<-d}$.
\end{lem}
\begin{proof}
  For objects $A,A'\in\A$ and $i\in\bbZ$ we have
  $\Ext^i(A,A')\cong\Hom(A,\Si^iA')$.  Thus the global dimension of
  $\A$ is bounded by $d$ if and only if for all objects
  $X,Y\in\bfD^b(\A)$ with cohomology concentrated in a single degree
  we have $\Hom(X,Y)=0$ when $X\in\D^{\ge 0}$ and $Y\in\D^{<-d}$.  The
  assertion of the lemma follows since for $X\in\D^{\ge 0}$ and
  $Y\in\D^{<-d}$, the truncations induce finite filtrations
  \[X=\t_{\ge 0}X\twoheadrightarrow \t_{\ge 1}X\twoheadrightarrow 
    \t_{\ge 2}X\twoheadrightarrow   \cdots\]
  and
  \[\cdots \rightarrowtail \t_{< -d-2}Y \rightarrowtail \t_{< -d-1}Y \rightarrowtail \t_{< -d}Y=Y\]
  such that each subquotient has its cohomology concentrated in a
  single degree $i$, with $i\ge 0$ for the subquotients of $X$ and $i<-d$
  for the subquotients of $Y$.
  \end{proof}

  We wish to extend this lemma from $\bfD^b(\A)$ to $\bfD(\A)$. To
  this end, we fix a Grothendieck category $\A$. Let us recall some
  basic facts about derived limits and colimits in $\bfD(\A)$. We will
  use derived functors in the sense of Deligne \cite[1.2]{Deligne73}.
  Recall that one of the most pleasant properties of Deligne's definition
  is that for an adjoint pair $(F,G)$, if the derived functors exist,
  they still form an adjoint pair $(\bfL F, \bfR G)$, see for
  example section~13 of \cite{Keller96}. Let $I$ denote a small
  category. We write $\A^I$ for the Grothendieck category of functors $I \to \A$.
  The diagonal functor $\Delta\colon \A \to \A^I$ taking an object to the
  constant functor has a left adjoint $\colim$ and a right adjoint $\lim$.
  Let us examine their derived functors. Since the functor $\Delta$ is
  exact, its left and right derived functors exist and are canonically
  isomorphic to the induced functor $\bfD(\A)\to \bfD(\A^I)$. 
  For general $I$, the existence of $\Lcolim$ is
  unclear but if $I$ is filtered, then $\colim$ is exact (by the definition
  of a Grothendieck category) and so its left derived functor exists
  and is canonically isomorphic to the induced functor $\bfD(\A^I) \to \bfD(\A)$,
  which we still denote by $\colim$. This implies in particular that
  arbitrary coproducts exist in $\bfD(\A)$ and are computed by
  coproducts in the category of complexes. For arbitrary $I$, the
  category $\A^I$ is still a Grothendieck category. This implies that
  the right derived functor $\Rlim$ exists and is computed as
  $\Rlim X = \lim \bfi X$, where $X \to \bfi X$ is a homotopy injective
  resolution in the homotopy category of $\A^I$, see for example
  Theorem~14.3.4 in \cite{KashiwaraSchapira05}. In particular, this
  implies that products of arbitrary set-indexed families $(X_i)_{i\in I}$
  of objects of $\bfD(\A)$ exist in $\bfD(\A)$ and are computed as 
  products in the category of complexes
  \[
  \prod_{i\in I} \bfi X_i \; ,
  \]
  where $X_i \to \bfi X_i$ is a homotopy injective resolution for
  each $i\in I$. For example, if $X_i$ is homologically left bounded, then
  for $\bfi X_i$, we may take any strictly left bounded complex with
  injective components quasi-isomorphic to $X_i$.

\begin{lem}\label{le:derived-co-limits}
  For each complex $X\in\bfD(\A)$ its truncations induce exact triangles
\[
\begin{tikzcd}
  \Si^{-1} X \arrow{r} &\coprod\limits_{p\ge 0} \t_{\le p}X \arrow{r}
  & \coprod\limits_{p\ge 0} \t_{\le p}X \arrow{r} &X
\end{tikzcd}
\]  
and
\[
\begin{tikzcd}
\Rlim\limits_{q\le 0} \t_{\ge q}X  \arrow{r}
&\prod\limits_{q\le 0} \t_{\ge q}X
\arrow{r} & \prod\limits_{q\le 0} \t_{\ge q}X \arrow{r} &
\Si\big(\Rlim\limits_{q\le 0} \t_{\ge q}X\big).
\end{tikzcd}
\]
Moreover, we have $X\iso \Rlim \t_{\ge q}X$ when the injective
dimension of each $H^nX$ admits a global bound not depending on $n$ and
$H^nX=0$ for $n\gg 0$.
\end{lem}
\begin{rem} As explained above, the coproducts in the first triangle
may be computed in the category of complexes, whereas the
products in the second triangle are products in the derived category
or equivalently derived products computed using homotopy injective
resolutions. The triangles exist by Proposition A.5 (3) of \cite{KellerNicolas13} 
and its dual but we give a direct proof for the special case at hand below.
\end{rem}
\begin{proof}
  For the first triangle we observe that the colimit of the
  $\t_{\le p}X$ in the category of complexes can be computed
  degreewise.  This gives an exact sequence
\[
\begin{tikzcd}
  0 \arrow{r} &\coprod\limits_{p\ge 0} \t_{\le p}X \arrow{r} &
  \coprod\limits_{p\ge 0} \t_{\le p}X \arrow{r} & X\arrow{r}&0
\end{tikzcd}
\]  
of complexes and therefore an exact triangle in $\bfD(\A)$, as in the
assertion of the lemma.

For the second triangle we need to construct a K-injective (homotopy injective)
resolution of $(\t_{\ge q}X)$ in the category of complexes of inverse
systems. For each $q< 0$, choose an injective
resolution $H^q X \to J_q$. Then choose a K-injective
resolution $\t_{\ge 0}X\to I_0$ and, for $q<0$, recursively define morphisms
$\e_q \colon I_{q+1} \to \Sigma^{q+1}J_q$ such that we have morphisms
of triangles in $\bfD(\A)$
\[
\begin{tikzcd}
  \Sigma^{q} H^{q} X\arrow{r}\arrow{d}&\tau_{\geq q} X\arrow{r}\arrow{d}&
  \tau_{\geq q+1} X \arrow{r}\arrow{d}&\Si^{q+1} H^{q} X\arrow{d}\\
\Si^{q} J_{q}\arrow{r}&I_{q} \arrow{r}& I_{q+1} \arrow{r}{\e_q}&\Si^{q+1} J_q
\end{tikzcd}
\]
where the vertical morphisms are quasi-isomorphisms and $\Si I_q$ is
the cone over a lift to a morphism of complexes of $\e_q$.  The system
$(I_q)$ is then quasi-isomorphic to $(\tau_{\geq q} X)$ and
K-injective in the homotopy category of complexes of inverse systems.
Thus, it may be used to compute the right derived limit of
$(\tau_{\geq q} X)$. We obtain a degreewise split exact sequence of
complexes
\[
\begin{tikzcd}
0\arrow{r}&\lim I_q \arrow{r}
&\prod\limits_{q\le 0} I_q
\arrow{r} & \prod\limits_{q\le 0} I_q \arrow{r} &0
\end{tikzcd}
\]
and therefore an exact triangle in $\bfD(\A)$, as in the assertion of
the lemma, with
\[\Rlim \t_{\ge q}X \cong \Rlim I_q\cong\lim I_q .\]

Now suppose that the injective dimension of $H^qX$ admits a global
bound, say $d$, and we may assume that $H^q X=0$ for all
$q>0$.  To show the isomorphism $X\iso\Rlim\t_{\ge q}X$ we modify the
above construction of a K-injective resolution of
$(\t_{\ge q}X)$ as follows.  For each $q\leq 0$, choose an injective
resolution $H^q X \to J_q$, where the components of $J_q$ vanish in
all degrees strictly greater than $d$. We put
$I_0=J_0$ and, for $q<0$, recursively define morphisms
$\e_q \colon I_{q+1} \to \Sigma^{q+1}J_q$ as before.  Again, the
system $(I_q)$ may be used to compute the right derived limit of
$(\tau_{\geq q} X)$. Since the $J_q$ are uniformly right bounded, the
system $(I_q)$ becomes stationary in each degree. This yields in
$\bfD(\A)$ the required isomorphism
\[
X \cong \lim I_q \cong \Rlim (\tau_{\geq q} X).\qedhere
\]
\end{proof}

\begin{lem}\label{le:Grothendieck-gldim-d}
  Let $(\D^{\le 0},\D^{>0})$ denote the canonical t-structure on
  $\bfD(\A)$ and suppose the global dimension of $\A$ is bounded by
  $d$. Then for $X\in\D^{\ge 0}$ and $Y\in\D^{<-d-2}$ we have
  $\Hom(X,Y)=0$.
\end{lem}
\begin{proof}
  We apply \intref{Lemma}{le:derived-co-limits}. Thus $X$ fits into an
  exact triangle given by the truncations $\t_{\le p}X$, and it
  suffices to show that $\Hom(\t_{\le p}X, Y)$ and
  $\Hom(\Sigma\t_{\le p}X, Y)$ vanish for all $p$. On the other hand,
  $Y$ fits into an exact triangle given by the truncations
  $\t_{\ge q}Y$, and therefore it suffices to show that
  $\Hom(\t_{\le p}X, \t_{\ge q}Y)$,
  $\Hom(\Sigma \t_{\le p}X, \t_{\ge q}Y)$, and
  $\Hom(\Si \t_{\le p}X, \Si^{-1}\t_{\ge q}Y)$ vanish for all $p$ and $q$.
  This holds by \intref{Lemma}{le:can-t-structure-gldim} since both
  arguments belong to $\bfD^b(\A)$.
\end{proof}

\section*{Tilting for $\bfD(\A)$}

Let $\A$ be a Grothendieck category and $\bfD(\A)$ its unbounded
derived category. Recall that the category $\bfD(\A)$ has arbitrary
(set-indexed) coproducts given by coproducts in the category of
complexes. Notice that the right derived product functor yields
arbitrary products in $\bfD(\A)$. In particular, the product of a
family of left bounded complexes with injective components is also
their product in $\bfD(\A)$.

\begin{lem}\label{lemma:compact}
  If $C$ is a compact object of $\bfD(\A)$, then the cohomology $H^pC$
  vanishes for all but finitely many integers $p$.
\end{lem}

\begin{proof} For each $p\in\bbZ$, choose a monomorphism
$i_p\col H^pC \to I_p$ into an injective object. Using the identification
\[
\Hom_{\bfD(\A)}(C, \Sigma^{-p} I) = \Hom_\A(H^pC,I)
\]
valid for each injective $I$ of $\A$, the $i_p$ yield a morphism 
$i$ from $C$ to the product (in the category of complexes and in
the derived category) of the $\Sigma^{-p}I_p$. Clearly, in the category
of complexes (and hence in the derived category), this product is 
canonically isomorphic to the corresponding coproduct. So we obtain
a morphism from $C$ to the coproduct of the $\Sigma^{-p}I^p$ which
in cohomology induces the $i_p$. By the compactness of $C$, this morphism
factors through a finite subcoproduct of the $\Sigma^{-p}I_p$ so that
all but finitely many of the $i_p$ have to vanish. Since they are monomorphisms,
the same holds for the $H^pC$. 
\end{proof}

Now let $T$ be a \emph{tilting object} of $\bfD(\A)$. Thus $T$ is
compact, the group $\Hom(T,\Si^p T)$ vanishes for all $p\neq 0$, and
$\bfD(\A)$ equals its localizing subcategory generated by $T$.

Let $\La$ be the endomorphism ring of $T$. Then $\La$ is
quasi-isomorphic to the derived endomorphism algebra $\RHom(T,T)$ and
so the functor $\RHom(T,-)$ yields a triangle equivalence
\[
\bfD(\A)\longiso\bfD(\Mod \La),
\]
cf.~\cite{Ke1994}. We use it to identify $\bfD(\A)$ with $\bfD(\Mod\La)$.
The canonical t-structure on $\bfD(\A)$ is denoted by
$(\D^{\le 0},\D^{>0})$, while the canonical t-structure on
$\bfD(\Mod\La)$ is denoted by $(\D(\La)^{\le 0},\D(\La)^{>0})$.

\begin{lem}\label{le:finite-gldim}
  Suppose that $\A$ and $\Mod\La$ have finite global dimension. Then
the functor  $\RHom(T,-)$ restricts to an equivalence
  $\bfD^b(\A)\iso\bfD^b(\Mod \La)$.
\end{lem}
\begin{proof}
  Given objects $X,Y\in\bfD^b(\A)$ we have $\Hom(X,\Si^iY)=0$ for
  almost all $i$ since $\A$ has finite global dimension. This is
  easily shown by induction on the number of integers $n$ such that
  $H^n(X\oplus Y)\neq 0$. It follows that  $\RHom(T,-)$ restricts to
a functor   $F\colon \bfD^b(\A)\to\bfD^b(\Mod \La)$, since
\[H^i \RHom(T,X)\cong\Hom(T,\Si^iX)\] and $T\in\bfD^b(\A)$ by
\intref{Lemma}{lemma:compact}.  On the other hand, $\bfD^b(\Mod\La)$
equals the thick subcategory of $\bfD(\Mod\La)$ that is generated by
the category $\Proj\La$ of projective $\La$-modules, viewed as
complexes concentrated in degree zero, since $\La$ has finite global
dimension. It follows that $F$ is essentially surjective since $F$
identifies the closure of $T$ under arbitrary coproducts and direct summands
with $\Proj\La$.
\end{proof}

From now on suppose that the global dimension of $\A$ is bounded by $d$,
and fix $t\ge 0$ such that $H^pT=0$ for all $p\not\in[-t,0]$,
cf.~\intref{Lemma}{lemma:compact}.

\begin{lem}\label{le:tilt-t-structure-1}
  We have $\D(\La)^{\le 0}\subseteq \D^{\le 0}$.
\end{lem}
\begin{proof}
  For $X\in\D^{>0}$ and $i\le 0$ we have $\Hom(T,\Si^iX)=0$ since
  $T\in\D^{\le 0}$. It follows that $X\in\D(\La)^{>0}$, since
  $\bfD(\A)\iso\bfD(\Mod\La)$ identifies $T$ with $\La$ and
  $H^iX\cong\Hom(\La, \Si^iX)$ in $\bfD(\Mod\La)$. Thus
  $\D(\La)^{\le 0}\subseteq \D^{\le 0}$.
\end{proof}

\begin{lem}\label{le:tilt-t-structure-2}
  We have $\D(\La)^{\ge 0}\subseteq \D^{\ge -d-t-2}$.
\end{lem}
\begin{proof}
  Let $X\in\D^{\le 0}$. Then $H^iT=0$ for all $i\not\in[-t,0]$ implies
  $\Hom(T,\Si^iX)=0$ for all $i>d+t+2$ by \intref{Lemma}{le:Grothendieck-gldim-d}. It
  follows that $\D^{\le 0}\subseteq\D(\La)^{\le d+t+2}$, and therefore
  $\D(\La)^{\ge 0}\subseteq \D^{\ge -d-t-2}$.
\end{proof}

\begin{proof}[Proof of \intref{Theorem}{th:main2}]
  Let $X,Y\in\Mod\La$ and $i>2d+t+4$. Then
  \[X\in \D(\La)^{\ge 0}\subseteq \D^{\ge
      -d-t-2}\qquad\textrm{and}\qquad \Si^iY\in\D(\La)^{<
      -2d-t-4}\subseteq \D^{< -2d-t-4}\] by
  \intref{Lemmas}{le:tilt-t-structure-1} and
  \ref{le:tilt-t-structure-2}. It follows from
  \intref{Lemma}{le:Grothendieck-gldim-d}
  that \[\Ext^{i}(X,Y)=\Hom(X,\Si^iY)=0.\] Thus the global dimension
  of $\La$ is bounded by $2d+t+4$.  In order to improve this bound,
  observe that $\RHom(T,-)$ restricts to an equivalence
  $\bfD^b(\A)\iso\bfD^b(\Mod \La)$ by
  \intref{Lemma}{le:finite-gldim}. Then we compare t-structures on
  $\bfD^b(\A)$ and use \intref{Lemma}{le:can-t-structure-gldim}
  instead of \intref{Lemma}{le:Grothendieck-gldim-d}. It follows that
  the global dimension of $\La$ is bounded by $2d+t$.
\end{proof}
  
\section*{Tilting for $\bfD^b(\A)$}  

Let $\A$ be an abelian category and $T\in\bfD^b(\A)$ a tilting object;
recall this means $\Hom(T,\Si^i T)=0$ for all $i\neq 0$ and
$\bfD^b(\A)$ equals the thick subcategory generated by $T$. Set
$\La=\End(T)$ and denote by $\proj\La$ the category of finitely
generated projective $\La$-modules. By Theorem~3.2 of
\cite{KV1987}, the inclusion $\add T \hookrightarrow\bfD^b(\A)$
extends to a triangle functor $\bfK^b(\add T) \rightarrow \bfD^b(\A)$.
Then it is straightforward to show
that the composite $\proj\La\iso\add T\hookrightarrow\bfD^b(\A)$
extends to a triangle equivalence
\[\bfD^b(\proj\La)\longiso\bfK^b(\add T)\longiso\bfD^b(\A).\]

We deduce \intref{Theorem}{th:main} from \intref{Theorem}{th:main2} when
$\A$ is \emph{noetherian}, that is, each object in $\A$ is
noetherian. To this end, we fix an essentially small abelian category $\A$
and let $\bar\A:=\Lex(\A^\op,\Ab)$ denote the category of left exact
functors $\A^\op\to\Ab$. Then $\bar\A$ is a Grothendieck category and
the Yoneda embedding $\A\to\bar\A$ which sends $X\in\A$ to $\Hom(-,X)$
is fully faithful and exact, cf.~\cite[Chap.~II]{Ga1962}.

\begin{lem}\label{le:locally-noeth}
  Suppose that $\A$ is noetherian and of finite global dimension. Then
  $\bfD(\bar\A)$ is compactly generated (so equals the localizing
  subcategory generated by all compact objects) and the inclusion
  $\A\to\bar\A$ induces a fully faithful functor
  $\bfD^b(\A)\to\bfD(\bar\A)$ that identifies $\bfD^b(\A)$ with the
  full subcategory of compact objects.
\end{lem}
\begin{proof}
  The inclusion $\A\to\bar\A$ identifies $\A$ with the full
  subcategory of noetherian objects in $\bar\A$.  It is well-known
  that an object $I$ of $\bar\A$ is injective if and only if
  $\Ext^1(-,I)$ vanishes on all noetherian objects. This implies that the
  global dimension of $\bar\A$ equals that of $\A$.
  
  Let $\Inj\bar\A$ denote the full subcategory of injective objects
  and $\bfK(\Inj\bar\A)$ the category of complexes up to
  homotopy. Then the canonical functor
  $\bfK(\Inj\bar\A)\to\bfD(\bar\A)$ is an equivalence,
  cf.~\cite[Proposition~3.6]{Kr2005}. It follows that $\bfD(\bar\A)$
  is compactly generated and that $\bfD^b(\A)$ identifies with the
  full subcategory of compact objects,
  cf.~\cite[Proposition~2.3]{Kr2005}.
\end{proof}

\begin{proof}[Proof of \intref{Theorem}{th:main}]
  We apply \intref{Lemma}{le:locally-noeth}. The functor
  $\bfD^b(\A)\to\bfD(\bar\A)$ identifies a tilting object $T$ of
  $\bfD^b(\A)$ with a tilting object of $\bfD(\bar\A)$. Let
  $\La=\End(T)$.  Then \intref{Theorem}{th:main2} provides the bound
  for the global dimension of $\La$. When $\La$ is right coherent,
  then the triangle equivalence $\bfD(\bar\A)\iso\bfD(\Mod \La)$
  restricts to an equivalence
  \[\bfD^b(\A)\iso\bfD^b(\proj \La)\iso\bfD^b(\mod \La)\]
  on the full subcategory of compact objects 
\end{proof}

\section*{Concluding remarks}

We end this paper with some remarks. Let us fix an essentially
small abelian category $\A$  with a tilting object $T\in\bfD^b(\A)$,
and set  $\La=\End(T)$.

Recall that a $\La$-module $X$ is \emph{pseudo-coherent} if it admits
a projective resolution
  \[\cdots \lto P_1\lto P_0\lto X\lto 0\] such that each $P_i$ is
  finitely generated. We denote by $\pcoh\La$ the full subcategory of
  pseudo-coherent $\La$-modules; it is a thick subcategory of the
  category of all $\La$-modules, so closed under direct summands,
  extensions, kernels of epis, and cokernels of monos.

\begin{rem}  
  Suppose that $\A$ is noetherian and of finite global dimension. Then
  $\RHom(T,-)$ induces a triangle equivalence
  $\bfD^b(\A)\iso \bfD^b(\pcoh\La)$.
\end{rem}

For each pair of objects $X,X'\in\A$ we have $\Ext^i(X,X')=0$ for
$i\gg 0$. This provides some restriction on the global dimension of
$\A$.

\begin{rem}
  Let $\A$ be a \emph{length category}; thus each object has finite
  composition length. Then
\[\gldim\A=\inf_{\substack{S,S'\\ \textrm{simple}}}\{i\in\bbN\mid
  \Ext^{i+1}(S,S')=0\}<\infty\]
since the number of isoclasses of simple objects is bounded by the
length of $H^*T$.
\end{rem}

\begin{rem}
  The global dimension of $\A$ need not to be finite when $\bfD^b(\A)$
admits a tilting object. Let $\La$ be a right noetherian ring and set
$\A=\mod\La$. Then $\La\in\bfD^b(\A)$ is tilting if and only if each
object in $\A$ has finite projective dimension. In this case the
global dimension of $\A$ equals the (small) finitistic dimension of
$\La$, which may be infinite (even when $\La$ is commutative),
cf.~\cite[Appendix, Example~1]{Na1962}.
\end{rem}


\end{document}